\documentclass[12pt,english,a4paper]{amsart}
\usepackage[latin9]{inputenc}
\usepackage{a4}
\usepackage{amsfonts}
\usepackage{amssymb}
\usepackage{amsmath}
\usepackage{amsthm}
\usepackage{changepage}
\usepackage{nomencl}
\usepackage{geometry}
\usepackage{tikz}
\usepackage{verbatim}
\usetikzlibrary{matrix}
\usepackage[toc,page]{appendix}
\begin{document}
\providecommand{\keywords}[1]{\textbf{\textit{Keywords: }} #1}
\newtheorem{theorem}{Theorem}[section]
\newtheorem{lemma}[theorem]{Lemma}
\newtheorem{proposition}[theorem]{Proposition}
\newtheorem{corollary}[theorem]{Corollary}
\theoremstyle{definition}
\newtheorem{definition}{Definition}[section]
\theoremstyle{remark}
\newtheorem{remark}{Remark}
\newtheorem{conjecture}{Conjecture}
\newtheorem{question}{Question}

\newcommand{\cc}{{\mathbb{C}}}   
\newcommand{\ff}{{\mathbb{F}}}  
\newcommand{\nn}{{\mathbb{N}}}   
\newcommand{\qq}{{\mathbb{Q}}}  
\newcommand{\rr}{{\mathbb{R}}}   
\newcommand{\zz}{{\mathbb{Z}}}  

\title{On number fields with $k$-free discriminant}
\author{Joachim K\"onig}
\email{jkoenig@kaist.ac.kr}
\address{Department of Mathematical Sciences, Kaist, 291 Daehak-ro, Yuseong-gu, Daejeon (South Korea)}
\begin{abstract}
Given a finite transitive permutation group $G$, we investigate number fields $F/\mathbb{Q}$ of Galois group $G$ whose discriminant is only divisible by small prime powers. This generalizes previous investigations of number fields with squarefree discriminant.
In particular, we obtain a comprehensive result on number fields with cubefree discriminant.
Our main tools are arithmetic-geometric, involving in particular an effective criterion on ramification in specializations of Galois covers.
\end{abstract}
\maketitle

\section{Introduction and overview of main results}
An area of special interest in the inverse Galois problem is the construction of field extensions with ``light ramification".
This can be understood in several different ways, all of which are important in certain applications. Firstly, one may try to minimize the number of ramified primes in an extension of number fields with prescribed Galois group. This is known as the minimal ramification problem. General conjectures on the minimum for Galois extensions of $\qq$ have been proposed, and proven for some special classes of groups (e.g., semiabelian $p$-groups in \cite{KNS}).
Furthermore, one may try to minimize the occurring ramification indices. This is studied in generality in \cite{KoeNS}.
Both of the above lines of research are also closely related to the investigation which groups occur as Galois groups of unramified Galois extensions over certain classes of number fields (e.g., of quadratic number fields).

In this article, we investigate a different minimization problem, namely minimization of the exponents occurring in the discriminant of a number field.
In particular, given a finite transitive group $G\le S_n$, call a finite separable (but in general non-Galois) field extension $F/K$ a $G$-extension if the Galois group of its Galois closure is permutation-isomorphic to $G$. We are interested in the smallest $k\in \mathbb{N}$, such that there exists a $G$-extension of $\qq$ with $k$-free discriminant; that is, whose discriminant is not divisible by any prime power $p^k$ (for given $k\in \mathbb{N}$).
The special case $k=2$, i.e., number fields with squarefree discriminant, has been investigated in-depth for several decades (e.g., \cite{Uchida}, \cite{Yamamoto}, \cite{Ked}, \cite{Bhargava}), including the stronger requirement of squarefree {\textit{polynomial}} discriminants. By an easy group-theoretical exercise, the Galois group of (the Galois closure of) a Galois extension with squarefree discriminant must be the full symmetric group. 
Much less is known, however, for $k\ge 3$.

The purpose of this paper is twofold: We give a general criterion (Theorem \ref{thm:chevalley_weil}), reducing the construction of number fields with prescribed inertia subgroups to function field extensions with such inertia subgroups, under mild requirements, using the theory of specializations of Galois covers. In Section \ref{sec:proof}, we then apply this criterion to the construction of extensions with $k$-free discriminant for smallest possible $k$, for certain classes of groups.
In particular, for the special case $k=3$, we cover all possible Galois groups, providing many  extensions with cubefree discriminant for each finite group which does not have a trivial obstruction to being the Galois group of such an extension (see Theorem \ref{thm:cubefree}). We also sketch how to obtain an analogous result in the case $k=4$ in the appendix (Section \ref{sec:append}). Our arguments are especially suited for Galois groups which are close to a full wreath product, via a careful construction of composition of covers with prescribed ramification types and prescribed specialization behaviours (see especially Lemmas \ref{lem:wreath}, \ref{a2n_intersect} and \ref{lem:last}).
In Section \ref{sec:variants}, we discuss some related problems of interest.

\section{Discriminants and indices in permutation groups}
\begin{definition}
Let $G\le S_n$ be a permutation group of degree $n$. For $\sigma\in G$, the index $ind(\sigma)$ is defined as $n$ minus the number of disjoint cycles of $\sigma$.\\
We define the generator index of $G$ as 
$$gi(G):=\min_{S\subset G: G = \langle S \rangle} \max_{\sigma \in S} ind(\sigma).$$
\end{definition}

The relevance of the above definition in Galois theory is due to the following easy observation. 
\begin{proposition}
Let $F/\mathbb{Q}$ be a degree-$n$ extension with Galois group $G\le S_n$, and let $e$ be the highest exponent of any prime number occurring in the prime factorization of the discriminant $\Delta(F)$.
Then $gi(G) \le e$, and if equality holds, then $F/\mathbb{Q}$ is tamely ramified.
\end{proposition}
\begin{proof} This only requires two well-known facts: that the exponent of a prime $p$ in the discriminant of $F$ equals the index of the inertia group generator in the case of tame ramification (and is larger than the index of any inertia group element in case of wild ramification), and that the normal closure of the set of all inertia groups equals the whole Galois group $G$.
\end{proof}

An obvious question is whether the converse holds, i.e., 
\begin{question}
\label{ques:small_disc}
Let $G$ be a transitive degree-$n$ permutation group of generator index $e$. Does there always exist a tame degree-$n$ $G$-extension $F/\mathbb{Q}$ whose discriminant is $(e+1)$-free (i.e., not divisible by any prime power $p^{e+1}$)?
\end{question}

Special cases of this question lead to connections with several problems of interest. We point out a few of those:
\begin{remark}
\begin{itemize}
\item[a)]
It is an easy exercise to show that a transitive permutation group of generator index $1$ is necessarily a symmetric group. In this case, it is classical that there are many $S_n$-extensions $F/\mathbb{Q}$ of degree $n$ with squarefree discriminant. Question \ref{ques:small_disc} is the natural generalization of this case to a problem depending on a prescribed group.
\item[b)] Let $G=D_n$ be the dihedral group of degree $n$. Then $gi(G) = \begin{cases}\frac{n-1}{2}, n \text{ odd} \\ \frac{n}{2}, n \text{ even} \end{cases}$. More precisely, a generating set attaining the bound $gi(G)$ is the set of reflections in $D_n$.
Question \ref{ques:small_disc} then reduces to the question whether there exist $D_n$-Galois extensions $E/\qq$ (unramified at $2$) such such that $E$ is unramified over the quadratic subextension fixed by the rotations in $D_n$. This is known to have a positive answer by class field theory, since there are many quadratic number fields (unramified at $2$) whose class number is divisible by $n$.
\item[c)] For a group generated by involutions (e.g., a non-abelian simple group) in its regular permutation action, Question \ref{ques:small_disc} is equivalent to the question whether there exist tame Galois extensions of $\mathbb{Q}$ with group $G$ all of whose non-trivial inertia groups are of order $2$. Cf.\ e.g. \cite{KRS} for investigation of such extensions.
\end{itemize}
\end{remark}

It should also be pointed out that the obvious analog of Question \ref{ques:small_disc} over the field $\cc(t)$ instead of $\qq$ has a positive answer due to Riemann's existence theorem.

For our criteria, as well as for certain inductive arguments, the following strong version of Question \ref{ques:small_disc} is useful:
\begin{question}
\label{ques:small_disc_strong}
Let $G$ be a transitive degree-$n$ permutation group of generator index $e$. Is it true that, for any positive integer $N$, there exists a degree-$n$ $G$-extension $F/\mathbb{Q}$ whose discriminant is $(e+1)$-free and coprime to $N$?\end{question}

Below is a first observation about Question \ref{ques:small_disc_strong}, regarding compatibility with direct products. To this end, let $G$ and $H$ be transitive groups of degrees $m$ and $n$. 
Then the direct product $G\times H$ naturally becomes a transitive group of degree $mn$. 
\begin{lemma}
\label{lem:dir_prod}
If Question \ref{ques:small_disc_strong} has a positive answer for $G$ and for $H$, then also for $G\times H$. In particular, it has a positive answer for all abelian groups.
\end{lemma}
\begin{proof}
If $x\in G$ consists of disjoint cycles of lengths $e_1,...,e_r$ and $y\in H$ of cycles of lengths $f_1,...,f_s$, then it is an easy exercise to show that $(x,y)$ has exactly the cycle lengths $lcm(e_i,f_j)$ repeated $gcd(e_i,f_j)$ times, for each $1\le i\le r$ and $1\le j\le s$.  In particular, the smallest index of an element of $G\times H$ projecting to $x\in G$ is the one of $(x,1)$ itself. Since any generating system of $G\times H$ projects to a generating system of $G$ and of $H$, the generator index of $G\times H$ is therefore given as $\max\{n \cdot gi(G), m\cdot gi(H)\}$.
Therefore, if $E/\qq$ and $F/\qq$ are degree-$n$ and -$m$ extensions answering Question \ref{ques:small_disc} in the positive for $G$ and for $H$ respectively, and assumed without loss to have discriminants coprime to some given $N$ and to each other, then $EF/\mathbb{Q}$ yields a positive answer for $G\times H$ with discriminant coprime to $N$, by standard discriminant formulae.
The conclusion about abelian groups follows since these have only one transitive permutation action (the regular one), and the above reduces the problem to cyclic groups of prime power order $p^d$, for which any tame Galois realization yields a positive answer to Question \ref{ques:small_disc} (since the generator index equals $p^d-1$). 
\end{proof}

We now turn to groups of generator index $2$.
\begin{proposition}
\label{prop:gen_index2}
Let $G$ be a transitive permutation group of generator index $\le 2$. Then one of the following holds:
\begin{itemize}
\item[(1)] $G = A_n$ or $G=S_n$, in their natural degree-$n$ permutation action.
\item[(2)] $G = C_2 \wr S_n$, the imprimitive wreath product of permutation degree $2n$.
\item[(3)] $G = (C_2\wr S_n) \cap A_{2n} \le C_2\wr S_n$, with the same degree $2n$ action as in (2).
\item[(4)] $G$ one of $D_5$, $PSL_2(5)$, $PSL_3(2)$ and $AGL_3(2)$ in their natural permutation actions of degree $5$, $6$, $7$ and $8$ respectively.
\end{itemize}
\end{proposition}
\begin{proof}
$S_n$ is the only transitive group generated by transpositions, so we may assume without loss that $gi(G) = 2$. Assume first that $G$ acts imprimitively on $n$ blocks of size $d$. Firstly, the block size $d$ must be $2$, since any generating set of $G$ must contain an element permuting at least $2$ blocks, and such an element has to move at least $2d$ points. Next, such an element can only be of index $\le 2$ if it is a double transposition, and then (since it switches exactly two blocks) its image in the action on the $n$ blocks must be a transposition. Therefore the image of $G$ in the blocks action is generated by transpositions, i.e., must be the full symmetric group $S_n$. So $G\le C_2\wr S_n$. Note that this group does not contain $3$-cycles, so $G$ must be generated by double transpositions, and possibly transpositions. This is of course the case for the full wreath product $C_2\wr S_n$, since the socle $C_2^n$ is generated by transpositions, and the complement $S_n\le C_2\wr S_n$ by double transpositions.
 Assume now that $G$ is smaller than the full wreath product $C_2\wr S_n$. Then $G$ cannot contain any transpositions, since any transposition generates a component $C_2\le C_2^n$, and transitivity in the action on the blocks would imply $C_2^n\le G$. 
 So $G\le (C_2\wr S_n) \cap A_{2n}$ is generated by double transpositions. If $G$ were strictly smaller than $(C_2\wr S_n) \cap A_{2n}$, then its intersection with $C_2^n$ would be a non-trivial submodule of the permutation module $C_2^n$ of dimension strictly smaller than $n-1$. There is only one such module, namely the diagonal one of dimension $1$ (see e.g. \cite{Mortimer}). Now $G$ contains a set of double transpositions which, when projected onto $S_n$, generate all of $S_n$. These elements are then necessarily of the form $(1,x)\in (C_2)^n\rtimes S_n$, whence $G$ contains a complement $\cong S_n$ to $C_2$, i.e., $G\cong C_2\times S_n$, acting on cosets of $\{1\}\times S_{n-1}$. But then any generating set of $G$ must contain an element projecting onto the generator of $C_2$, and such an element is  necessarily fixed point free, i.e., of index at least $n\ge 3$, a contradiction.
 
 So assume now that $G$ is a primitive group. It is well-known that, if $G$ contains a transposition, then it must be the full symmetric group, and if it contains a $3$-cycle, then it must contain the alternating group. We may therefore assume that $G$ is generated by double transpositions. The fact that any non-alternating primitive group with this property is contained in case (4) is shown e.g. in \cite{Kondo97}.
 \end{proof}

Using Proposition \ref{prop:gen_index2}, we provide many extensions with cubefree discriminant, for each group which can possibly occur as the Galois group of (the Galois closure of) such an extension:

\begin{theorem}
\label{thm:cubefree}
Let $G$ be a transitive permutation group of degree $n$ and of generator index $g\le 2$, and let $S$ be a finite set of prime numbers. Then there exist infinitely many $G$-extensions $F/\mathbb{Q}$ of degree $n$ such that
\begin{itemize}
\item[i)] $\Delta(F)$ is $(g+1)$-free.
\item[ii)] $F/\mathbb{Q}$ is unramified at all primes in $S$.
\end{itemize}
More precisely, there exists a positive constant $\alpha:=\alpha(G)$ such that for all sufficiently large $N\in \mathbb{N}$, the set of all $G$-extensions $F/\mathbb{Q}$ with properties i) and ii), and with $|\Delta(F)|\le N$, is of cardinality at least $N^{\alpha}$.
\end{theorem}

The proof of Theorem \ref{thm:cubefree} is contained in Section \ref{sec:proof}. As for the ``true" exponent $\alpha$ in Theorem \ref{thm:cubefree}, we conjecture it to be $=1/2$ or $=1$, depending on whether $G$ is or is not contained in $A_n$.

\section{Some prerequisites}
{\textit{Function field extensions}}. We briefly review some basic notions concerning specialization of function field extensions, which will feature in the following sections.

Let $k$ be a field of characteristic zero, and let $E/k(t)$ be a degree-$n$ extension.
Recall that $E/k(t)$ is called $k$-regular if $E\cap \overline{k} = k$. For any transitive group $G\le S_n$, equivalence classes of $k$-regular degree-$n$ $G$-extensions $E/k(t)$ correspond one-to-one to equivalence classes of 
connected degree-$n$ covers $f:X\to \mathbb{P}^1$, defined over $k$ and with Galois group (of the Galois closure) permutation-isomorphic to $G$.
Given a  Galois cover $E/k(t)$ and a $k$-rational place $t\mapsto t_0\in \mathbb{P}^1(k)$, we denote by $E_{t\mapsto t_0}$ (or simply $E_{t_0}$) the residue extension at any place extending $t\mapsto t_0$ in $E$. This is independent of the choice of extending place, since $E/k(t)$ is Galois. We explicitly also allow residue extensions at branch points of $E/k(t)$.

{\textit{Residue extensions and completions}}, Throughout, we use some basic and well-known facts about residue extensions and completions in number field and function field extensions. We refer to \cite{Serre} for details. Let us only point out the following well-known fact: Let $E/k(t)$ be a Galois extension over a number field $k$, with ramification index $e\ge 1$ at a $k$-rational place $t\mapsto t_0$, and set $s:=t-t_0$. Then the completion $E\cdot k((s))$ of $E$ at any place extending $t\mapsto t_0$ is of the form $E_{t\mapsto t_0}(((\alpha s)^{1/e}))$, with some $\alpha\in E_{t\mapsto t_0}$.

{\textit{Ramification in specializations}}. In Section \ref{sec:crit}, we investigate specializations $E_{t_0}/\qq$ with $k$-free discriminant by comparing ramification in specializations with ramification in $E/\qq(t)$. A well-known theorem in this direction was given by Beckmann in \cite[Proposition 4.2]{Beckmann}. We review it here, and give a strong version in Theorem \ref{thm:chevalley_weil} below. The theorem remains true over number fields, but most of our applications will be over $\qq$, in which case the statement becomes particularly nice.

\begin{proposition}
\label{prop:beckmann}
Let $E/\mathbb{Q}(t)$ be a $\qq$-regular Galois extension with group $G$, branch points $t_1,...,t_r\in \mathbb{P}^1(\overline{\qq})$ and associated inertia group generators $\sigma_1,...,\sigma_r\in G$. 
For $i=1,\dots,r$, denote by $\mu_{t_i}(X,Y)\in \zz[X,Y]$ the homogenization of the irreducible polynomial of $t_i$ in $\zz[X]$.\footnote{With $\mu_{\infty}$ defined as $Y$.}
Then there exists a finite set $S_0$ of prime numbers (``bad primes" of $E/\mathbb{Q}(t)$), such that for all $p\in \mathbb{P}\setminus S_0$, the following holds:\\
For any $t_0 =[a:b] \in \mathbb{P}^1(\qq)\setminus\{t_1,\dots t_r\}$ (with $a,b\in \zz$ not both $=0$), $p$ ramifies in $E_{t_0}/\qq$ if and only if there exists a branch point $t_i$ (automatically unique up to algebraic conjugates, via appropriate choice of $S_0$) such that $\mu_{t_i}(a,b)\in \zz$ is of $p$-adic valuation $\nu>0$, not divisible by $| \langle \sigma_i\rangle |$. Furthermore, in this case the inertia group at $p$ in $E_{t_0}/\qq$ is conjugate in $G$ to the subgroup of $\langle\sigma_i\rangle$ of index $gcd(\nu,|\langle\sigma_i\rangle |)$.
\end{proposition}

\section{Specializations with $k$-free discriminant: a general criterion}
\label{sec:crit}
In Theorem \ref{thm:chevalley_weil} below, we give a strong version of Beckmann's criterion (Proposition \ref{prop:beckmann}), relating inertia groups in a given $\qq$-regular extension with those in its specializations, under a mild extra assumption. For our applications, we cannot allow any exceptional primes as in the statement of Proposition \ref{prop:beckmann}. We therefore make the following assumption, where $E/k(T)$ is a finite $k$-regular Galois extension over a number field $k$.

 (*) For each finite prime $p$ of $k$, there exists a point $t_0(p)\in \mathbb{P}^1(k)$ (possibly a branch point!) such that 
 $p$ does not divide the ramification index at $t_0$ and
 the residue extension $E_{t_0(p)}/k$ is unramified at $p$. 
 
\begin{theorem}
\label{thm:chevalley_weil}
Let $k$ be a number field, $S$ be a finite set of primes of $k$ and $E/k(t)$ be a finite Galois extension fulfilling condition (*), with branch points $t_1,..,t_r\in \mathbb{P}^1(\overline{k})$ and associated inertia group generators $\sigma_1,...,\sigma_r$. 
Denote by $\mathcal{S}(E,\{\sigma_1,...,\sigma_r\},S)$ the set of all $t_0\in k\setminus \{t_1,...,t_r\}$ such that  $E_{t_0}/k$ is unramified in $S$ and all its inertia groups are conjugate to a subgroup of some $\langle\sigma_i\rangle$, $i\in \{1,...,r\}$.
Then \begin{itemize}
\item[a)] there exists a finite set $S_0$ of primes of $k$ (depending only on $E/k(t)$) such that $\mathcal{S}(E,\{\sigma_1,...,\sigma_r\}, S\cup S_0)$ contains a non-empty and $(S\cup S_0)$-adically open subset of $\mathbb{P}^1(k)$.
\item[b)] If additionally $t_1$ is $k$-rational, then $\mathcal{S}(E,\{\sigma_2,...,\sigma_r\}, S\cup S_0)$ is infinite.
\end{itemize}
\end{theorem}
\begin{proof} 
Choose $S_0$ as in Proposition \ref{prop:beckmann}.
For each $p\in S\cup S_0$, choose a value $t_0(p)$ as in the assumption. Assume first that $t_0(p)$ is a non-branch point. Then for all $t_0\in k$ which are sufficiently close $p$-adically to $t_0(p)$ (for all $p\in S\cup S_0$), it follows that $E_{t_0}/k$ is unramified inside $S\cup S_0$. This makes use a variant of Krasner's lemma; compare e.g. Lemma 3.5 in \cite{BSS}. By Proposition \ref{prop:beckmann}, any such $t_0$ now fulfills $t_0\in \mathcal{S}(E,\{\sigma_1,...,\sigma_r\}, S\cup S_0)$. Clearly the set of such $t_0$ is $(S\cup S_0)$-adically open, showing a).\\
If $t_1$ is additionally $k$-rational, then apply a fractional linear parameter transformation on $k(t)$ to map $t_1$ to $\infty$.
Then by the above construction $\mathcal{S}(E,\{\sigma_1,...,\sigma_r\},S\cup S_0)$ contains a generalized arithmetic progression $a+bO_k$, with suitable $a\in k$, $b\in O_k$. Note here that $a$ can be assumed $q$-integral\footnote{That is, contained in the localization of $O_k$ at $q$.} for all primes $q\notin S\cup S_0$: indeed, using Krasner's lemma, the set of non-branch points $t_0(p)$ fulfilling (*) is $p$-adically open and thus contains elements which are $q$-integral for all primes $q\ne p$. Since a $q$-integral value $t_0$ cannot intersect infinity at $q$, we have $a+bO_k \subset \mathcal{S}(E,\{\sigma_2,...,\sigma_r\}, S\cup S_0)$, showing b).

Finally, assume that $t_0(p)$ as above is a $k$-rational branch point of ramification index $e>1$, without loss $t_0(p) = 0$. Let $F/k$ be the residue extension at $t\mapsto 0$.
Then the completion at $t\mapsto 0$ equals $F((\sqrt[e]{\alpha t}))$ for some $\alpha\in F$ (note that automatically, $\zeta_e\in F$).  By assumption, $p$ does not ramify in $F/k$. Now let $s$ be a root of $X^e-at$, with $a\in k^\times$ to be determined. Then the  completion at $t\mapsto 0$ of the Galois closure $k(\zeta_e)(s)/k(t)$ equals $k(\zeta_e)((\sqrt[e]{at}))$. In particular, the composite of both completions is $F((t))(\sqrt[e]{at},\sqrt[e]{\alpha/a}) = F(\sqrt[e]{\alpha/a})((s))$, and thus the residue extension of $E(s)/k(s)$ at $s\mapsto 0$ is contained in $F(\sqrt[e]{\alpha/a})$ (note here that $E(s)/k(s)$ is unramified at $s\mapsto 0$ due to Abhyankar's lemma).
It suffices to ensure that $p$ does not ramify in $F(\sqrt[e]{\alpha/a})/k$; since $F/k$ is already unramified at $p$ 
and $p$ does not divide $e$ by assumption, 
a sufficient condition for this is that $\alpha/a$ is of $p$-adic valuation $0$. 



But of course, the $p$-adic valuation of $\alpha$ is an integer (since $k(\alpha)/k$ is unramified at $p$), so this can be achieved with $a\in k^\times$. Since $E(s)/k(s)$ is unramified at $0$, the above arguments show that there exists a $p$-adically open neighborhood of $s\mapsto s_0$ around $0$ whose residue extensions in $E(s)/k(s)$ are also unramified at $p$. The same must then hold for the corresponding specializations $E_{t\mapsto s_0^e/a}$, where $t_0=s_0^e/a$ can be assumed to be a non-branch point of $E/k(t)$ without loss.
%
We have therefore reduced the problem to the case of non-branch points $t_0(p)$, which has already been treated. 
\end{proof}

Theorem \ref{thm:chevalley_weil} arose essentially from joint work in preparation of the author with Neftin and Sonn (\cite{KoeNS}) on unramified extensions of number fields. We also refer to this source for more precise conjectures concerning field counting results in the spirit of the next theorem below.
A special case of our criterion (namely, under the stronger assumption of existence of a completely split fiber) occurs as the ``Absolute Chevalley-Weil theorem" in \cite{Bilu}.
Our assumption (*) occurs similarly (as the case $U \subset \{\infty\}$) in \cite[Theorem 1.4]{BSS}, although with a different goal in mind (namely, minimizing the {\textit{number}} of ramified primes in specializations), and without the possibility of residue fields at branch points in condition (*).
We deliberately included the case of residue extensions at branch points, even though through the proof of Theorem \ref{thm:chevalley_weil} it reduces to the case of non-branch points. The advantage is that in many applications, residue fields at branch points are particularly easy to recognize. E.g., if some non-Galois cover of degree $n$ comes with a totally ramified $k$-rational branch point, then its residue extension in the Galois closure is exactly $k(\zeta_n)/k$. \\
Note finally that in the case $k=\qq$, the technical condition in (*) that $p$ should not divide the ramification index $e$ at $t_0$ is relevant only for $p=2$ (since any odd prime dividing $e$ would necessarily ramify in $\qq(\zeta_e)\subset E_{t_0}$). Even in this case, the condition could be somewhat relaxed; however, we prefer not to complicate the condition, which is sufficient for all our applications.

\begin{remark}
\label{rem:monodromy}
It is worth pointing out the following convenient fact: If $k=\mathbb{Q}$, then the Galois group of an extension $E/k(t)$ as in Theorem \ref{thm:chevalley_weil} is the group generated by $\sigma_1,...,\sigma_n$. This is true in general only after constant extension from 
$\mathbb{Q}$ to $\overline{\qq}$; however, the non-existence of finite primes ramifying in all specializations automatically implies that $E$ is regular over $\mathbb{Q}$ (i.e., $\overline{\mathbb{Q}}\cap E = \qq$). This observation often enables an entirely combinatorial computation of the Galois group of $E/\qq(t)$ (e.g., in terms of cycle structures).
\end{remark}

Combining Theorem \ref{thm:chevalley_weil} with field counting results, we obtain the following strengthening: 
\begin{theorem}
\label{cor:count}
Let $G\le S_n$, and assume that there exists a degree-$n$ $G$-extension $E/\qq(t)$, with branch points $t_1,...,t_r$, associated inertia group generators $\sigma_1,...,\sigma_r$, and whose Galois closure fulfills condition (*). 
For $B\in \mathbb{N}$ and $S\subset \mathbb{P}$ finite, let $N(B,G,\{\sigma_1,...,\sigma_r\},S)$ denote the number of degree-$n$ $G$-extensions of $\qq$ with discriminant of absolute value $\le B$,
with all inertia groups contained in a conjugate of one of $\langle\sigma_1\rangle,..., \langle \sigma_r\rangle$, and unramified in $S$.\\
Then \begin{itemize}
\item[a)] there exists $\alpha:=\alpha(f)>0$, depending only on $G$ (and not on $S$) such that 
$$N(B,G,\{\sigma_1,...,\sigma_r\},S) \gg_B B^\alpha.$$
More precisely, this holds for any $\alpha< (\sum_{i=1}^r ind(\sigma_i))^{-1}$.
\item[b)] If $t_1$ is rational, then 
$$N(B,G,\{\sigma_2,...,\sigma_r\},S) \gg_B B^\alpha,$$
for any $\alpha< (\sum_{i=2}^r ind(\sigma_i))^{-1}$.
\end{itemize}
Moreover, these inequalities can be fulfilled by only counting specializations of $E/\qq(t)$.
\end{theorem}
\begin{proof}
Firstly, if $t_1$ is rational, then we may and will assume, via fractional linear transformation, that $t_1=\infty$. 
Furthermore, since $k=\qq$, the $(S\cup S_0)$-adically open set constructed in the proof of Theorem \ref{thm:chevalley_weil} contains an arithmetic progression $a + N\zz$ (with some $a\in \qq, N\in \nn$). Applying a suitable linear transformation in $t$ (fixing infinity), we can map this set onto $\zz$, and may therefore assume that the set of specialization values $t_0$ fulfilling assertion a) (resp., b)) of Theorem \ref{thm:chevalley_weil} contains the integers $\zz$.
We need to estimate the number of distinct $G$-extensions of bounded discriminant in this set of specializations (that this set of specializations does contain many $G$-extensions follows from Hilbert's irreducibility theorem).

By Theorem 3.1 in \cite{Bilu} (which is  a strong version of a theorem due to Dvornicich and Zannier), the set of distinct $G$-extensions arising from specializations of $E/\qq(t)$ at $t_0 \in \{1,...,B\}$ is asymptotically bounded from below by $B/\log(B)$.
At the same time, since all these specializations values $t_0\in \zz$ yield extensions $E_{t_0}/\qq$ unramified at all bad primes $p\in S_0$, any ramified prime $p$ must come from $t_0$ meeting some finite branch point $t_i$ at $p$, which translates to $p$ being a prime divisor of $\mu_{t_i}(t_0)$, with $\mu_{t_i}\in \zz[X]$ the irreducible polynomial of $t_i$ over $\zz$. The absolute value of the discriminant of  $E_{t_0}$ is therefore bounded from above by $(\prod \mu_{t_i}^{ind(\sigma_i)})(t_0)$, where
the product includes one representative for each set of algebraically conjugate finite branch points. Since this polynomial $\prod \mu_{t_i}^{ind(\sigma_i)}\in \zz[X]$ is of degree $\sum_{i=1}^r ind(\sigma_i)$ (resp., $\sum_{i=2}^r ind(\sigma_i)$, if $t_1 = \infty$) we have $|\Delta(\qq(f^{-1}(t_0)))| \le C\cdot B^{\sum ind(\sigma_i)}$ for $|t_0|\le B$. Renaming $C\cdot B^{\sum ind(\sigma_i)}$ as $B$, the assertion follows.
\end{proof}

A few remarks on the exponent $\alpha$ in Theorem \ref{cor:count}:
\begin{remark}
\begin{itemize}
\item[a)] 
By the Riemann-Hurwitz formula, the largest possible value for $\alpha$ is $\frac{1}{n-1}-\epsilon$, which is attained in the case where $E$ is of genus $0$, and $t_1$ is rational and totally ramified. In this case, upon assuming $t_1 = \infty$ without loss, $E$ is the root field of a polynomial $F(X)-t$, with $F\in \mathbb{Q}[X]$. On the other hand, it is known by deep results of Bhargava et al. (\cite{Bhargava}, \cite{Bhargava2}) that the true exponent in the special case of counting $S_n$-number fields with squarefree discriminant is at least $1/2 + 1/n$; and in the case $n\le 5$ actually equal to $1$, implying that a positive density of $S_n$-number fields have squarefree discriminant. This discrepancy is not surprising, since we only count number fields arising as specializations of one given cover.
\item[b)] A similar exponent $\alpha$ for the counting of specializations of a Galois cover (fulfilling additional local conditions) was also obtained in \cite{Debes17}, Theorem 1.1, although depending on the discriminant of defining {\textit{polynomials}} rather than the (in general, smaller) discriminant
of the function field extension itself.
\end{itemize}
\end{remark}

The following direct consequence of the above results is our application to Question \ref{ques:small_disc_strong}. It reduces the problem to a geometric problem, namely finding a cover with suitable ramification data (and with the technical extra condition of non-existence of primes ramifying in all specializations).
\begin{corollary}
\label{cor:appl}
Let $E/\qq(t)$ be a $G$-extension fulfilling Condition (*), $(t_1,...,t_r)$ and $(\sigma_1,...,\sigma_r)$ be as in Theorem \ref{cor:count}, and assume that $\max_{\stackrel{i\in \{1,...,r\}}{t_i \text{ finite}}}ind(\sigma_i) = gi(G)$. Then Question \ref{ques:small_disc_strong} has a positive answer for $G$.
\end{corollary}

Important cases of groups for which the assumptions of Corollary \ref{cor:appl} can be fulfilled are provided in the next section. One more important case should be pointed out here.
\begin{remark} If $G$ is abelian, then the rigidity method (see e.g.\ \cite{MM}) yields that any set of cyclic subgroups generating $G$ can be realized as the set of non-trivial inertia subgroups of some $\qq$-regular $G$-extension $E/\qq(t)$. Let $t_0\in \qq$ be any non-branch point whose residue extension has full Galois group $G$. Then a standard field-crossing argument (also known as ``twisting lemma", e.g.\ Lemma 3.1 in \cite{Debes17}) yields a $\qq$-regular $G$-extension $\tilde{E}/\qq(t)$ with the same branch points and inertia groups as $E/\qq(t)$ and with a trivial residue extension at $t_0$. We therefore regain the observation made in Lemma \ref{lem:dir_prod} that Question \ref{ques:small_disc_strong} has a positive answer for abelian groups, but this time with a geometric argument (useful for inductive conclusions, see e.g. Lemma \ref{lem:wreath}).
\end{remark}

\section{Proof of Theorem \ref{thm:cubefree}}
\label{sec:proof}
We divide the proof of Theorem \ref{thm:cubefree} into several lemmas, dealing with the cases (1), (2), (3) and (4) of Proposition \ref{prop:gen_index2}.

Even though the case $S_n$ is well-understood, we reproduce a proof which is useful for later inductive arguments.
\begin{lemma}
\label{lem:sn}
Theorem \ref{thm:cubefree} holds for $G=S_n$, $n\ge 2$.
\end{lemma}
\begin{proof}
By Corollary \ref{cor:appl}, it suffices to find a $\qq$-regular degree-$n$ extension $E/\mathbb{Q}(t)$, defined by a polynomial $F(t,X) = f(X)-t$, such that 
\begin{itemize}
\item[i)] $f = (X-\alpha_1)\cdots (X-\alpha_{n})$ is separable and completely split over $\mathbb{Q}$.
\item[ii)] All non-trivial inertia groups of $E/\mathbb{Q}(t)$ outside $t\mapsto \infty$ are generated by transpositions.
\end{itemize}
We may then directly apply Corollary \ref{cor:appl} since $E/\mathbb{Q}(t)$ has trivial residue extension at $t\mapsto 0$.\\
Note that many such $F$ exist. Indeed, for a generic monic degree-$n$ polynomial $f = X^n + s_1 X^{n-1} + ... + s_n$, the discriminant of $f(X)-t$ is a separable polynomial in $t$. 
Hence, upon specializing the $n$ unknown coefficients $\sigma_i$ of $f$ into $\mathbb{A}^n(\mathbb{Q})$, the set of all specializations rendering the discriminant inseparable is contained in a union of finitely many subvarieties of dimension $<n$ (i.e., is not Zariski-dense), whereas the rational values $(s_1,...,s_n)$ for which $f$ splits are the images of $\mathbb{Q}$-points under the $S_n$-cover $\mathbb{A}^n\to \mathbb{A}^n$ mapping $(\alpha_1,...,\alpha_n)$ to the tuple of its elementary-symmetric functions, i.e., in particular form a Zariski-dense set. Since the discriminant of a generic $(X-\alpha_1)\cdots (X-\alpha_{n})-t$ must then be squarefree, we have in fact obtained that the specialization values $(\alpha_1,...,\alpha_n)$ satisfying i) and ii) again form a dense open subset of $\mathbb{A}^n(\qq)$.
\end{proof}

Next, we deal with alternating groups. I am not aware of the following statement occurring in the literature, in spite of its relative similarity with the previous $S_n$-argument.
\begin{lemma}
\label{lem:an}
Theorem \ref{thm:cubefree} holds for $G=A_n$, $n\ge 3$.
\end{lemma}
\begin{proof}
First, let $n\ge 4$ be even.
Let $F(t,X) = f(X)-tg(X)$ with the following properties:
\begin{itemize}
\item[i)] $f$ is a separable degree-$n$ polynomial, splitting completely over $\qq$.
\item[ii)] $g$ is separable of degree $<n$.
\item[iii)] ${\textrm{Gal}}(F | \qq(t)) = A_n$, and all non-trivial inertia groups are generated by $3$-cycles.
\end{itemize}
There exist many such polynomials $F$, by a famous construction due to Mestre (\cite{Mes90}). Indeed, for a generic completely split polynomial $f = \prod_{i=1}^n (X-\alpha_i)$, there exists a polynomial $g$ of degree $<n$ such that
${\textrm{Gal}}(f-tg | \mathbb{Q}(\alpha_1,...,\alpha_n)(t))=A_n$ and all non-trivial inertia groups are generated by $3$-cycles. In particular, the discriminant of $f-tg$ is the square of a squarefree polynomial in $t$, and therefore remains so upon many specializations of the coefficients $\alpha_i$, namely once again a Zariski-dense open subset (whence we also find many such specializations preserving the Galois group $A_n$).
Now let $E$ be a splitting field of $F$ as above over $\qq(t)$. Then $E/\mathbb{Q}(t)$ has trivial residue field at the unramified place $t\mapsto 0$, by definition. In particular, no prime ramifies in all specializations of $E/\mathbb{Q}(t)$ at non-branch points. The claim thus follows from Corollary \ref{cor:appl}.

Finally, let $n$ be odd. Take a polynomial $F(t,X)$ as above, of even degree $n+1$, and observe that, if $x$ denotes a root of $F$ in $E$, then the fixed field of $A_n$ in $E$ is the rational function field $\mathbb{Q}(x)$. Since $t\mapsto 0$ is completely split in $E$,
any place $x\mapsto \alpha_i$ extending $t\mapsto 0$ in $\mathbb{Q}(x)$ is completely split in $E$, and of course all non-trivial inertia groups in $E/\mathbb{Q}(x)$ are still generated by $3$-cycles. Therefore, Corollary \ref{cor:appl} is applicable to $E/\mathbb{Q}(x)$ as well, yielding the assertion for odd $n$.
\end{proof}

We now treat imprimitive groups. The imprimitive groups occurring in Theorem \ref{thm:cubefree} are of a very restricted form; however, larger classes of imprimitive groups can be treated in the same way, which should be useful when extending our results to higher generator indices (compare Section \ref{sec:append}). The following lemma provides such a criterion, which includes as a special case the case $G=C_2\wr S_n$ in Theorem \ref{thm:cubefree}.

\begin{lemma}
\label{lem:wreath}
Let $G\le S_k$ be transitive, and assume that there exists a degree-$k$ extension $E/\mathbb{Q}(t)$ with group $G$, fulfilling the assumptions of  Corollary \ref{cor:appl}. Then there also exists such an extension (of degree $nk$) for the imprimitive wreath product $G\wr S_n$, for any $n\ge 2$.
\end{lemma}
\begin{proof}
We first verify the generator index of the wreath product. This is easy even for a general wreath product $G_1\wr G_2$: Of course, a generating set of $G_1\wr G_2$ must project to a generating set of $G_2$, so must contain an element moving at least $gi(G_2)$ blocks, which implies moving $\ge n_1\cdot gi(G_2)$ points. However, this bound can also be attained. Indeed, embed $G_2$ into $G_1\wr G_2$ such that $\sigma \in G_2\le S_{n_2}$ moves $i$ to $j$ if and only if its image in $G_1\wr G_2$ moves point $k$ of the $i$-th block to point $k$ of the $j$-th block, for all $k=1,...,n_1$, and embed $G_1$ as the first component of the normal subgroup $G_1^{n_1}$. Let $S_i$ be a generating set of $G_i$ attaining the bound $gi(G_i)$, for $i=1,2$, and let $\Sigma:=\{(\sigma,1,...1 ; 1) \in G_1^{n_1}\rtimes G_2 | \sigma \in S_1\} \cup \{(1,...,1; \sigma_2) | \sigma_2\in S_2\}$. Due to transitivity of $G_2$ on the blocks, this set generates $G$, and each of its elements has index $\le \max\{\underbrace{gi(G_1)}_{<n_1} , n_1\cdot gi(G_2)\} = n_1\cdot gi(G_2) = gi(G_1\wr G_2)$.

Now let $L/\mathbb{Q}(u)$ (resp., $\widehat{L}/\mathbb{Q}(u)$) be a root field (resp., splitting field) of a polynomial $f(X)-u$ with Galois group $S_n$, where $f(X) = \prod_{i=1}^n(X-\alpha_i)$ is as in Lemma \ref{lem:sn}.
Identifying a root of $f(X)-u$ in $L$ with the above parameter $t$, we have $L=\mathbb{Q}(t)$. 
%
%
Assume for the moment that $f(X)$ is chosen such that the following holds:

(**) No two branch points of $E/\mathbb{Q}(t)$ extend the same $u$-value, and no branch point of $E/\mathbb{Q}(t)$ maps to $u=0$ or to a finite branch point of $\mathbb{Q}(t)/\mathbb{Q}(u)$. 

That this assumption is without loss will be verified later.
Consider then the extension $E/\mathbb{Q}(u)$. Its Galois closure $\widehat{E}/\mathbb{Q}(u)$ has Galois group $\Gamma$ naturally embedding into $G\wr S_n$ and surjecting onto $S_n$. On the other hand, all inertia groups at points which ramify on the level of $E/\mathbb{Q}(t)$ are contained in one single component of $G$ inside the normal subgroup $G^n$ of $G\wr S_n$. Transitivity of $S_n$ on the components implies that $\Gamma$ contains all of $G^n$, and therefore equals the full wreath product. Also, by construction, all inertia group generators of $E/\mathbb{Q}(u)$ at finite branch points have index $\le k = gi(G\wr S_n)$.

We now claim that, up to suitable choice of $f(X)$, Condition (*) is fulfilled, which suffices to apply Corollary \ref{cor:appl}. We ensure this by considering only the residue extensions $\widehat{E}_{u\mapsto \infty}/\qq$ and $\widehat{E}_{u\mapsto 0}/\qq$.
%
The first equals $F(\zeta_{en})/\qq$, where $F:=E_{t\mapsto \infty}$ and $e\ge 1$ is the ramification index of $E/\qq(t)$ at $t\mapsto \infty$:
indeed, since $\qq(t)/\qq(u)$ is totally ramified at $u\mapsto \infty$, it follows that the completion of $E/\qq(u)$ at any place extending $u\mapsto \infty$ equals $F((u^{1/en}))/\qq((u))$. Since $\widehat{E}/\mathbb{Q}(u)$ is the Galois closure of $E/\qq(u)$, its completion is therefore the Galois closure of $F((u^{1/en}))/\qq((u))$, which is $F(\zeta_{en})((u^{1/en}))/\qq((u))$.
In particular, this implies that the whole residue extension $E_{u\mapsto \infty}/\qq = F(\zeta_{en})/\qq$ is fixed independently of the concrete choice of $f$. It then remains to ensure that none of the primes dividing $en$ or ramifying in $F/\qq$ are ramified in $\widehat{E}_{u\mapsto 0}/\qq$. Denote the set of those primes by $S$.
But $\widehat{E}_{u\mapsto 0}/\qq$ is the compositum of all extensions $E_{t\mapsto \alpha_i}/\qq$, $i=1,...,n$, due to $u\mapsto 0$ being completely split in $\widehat{L}$. Here the roots $\alpha_1$,..., $\alpha_n$ of $f(X)$ are still to be chosen.

Due to the assumptions on $E/\mathbb{Q}(t)$, we may assume (using Theorem \ref{thm:chevalley_weil}) that there exist $\alpha_1,...,\alpha_n\in \qq$ such that all of the extensions $E_{t\mapsto \alpha_i}/\qq$ are unramified inside $S$. In fact, upon including the bad primes of $E/\mathbb{Q}(t)$ in the set $S$, Theorem \ref{thm:chevalley_weil} guarantees that the set of these $(\alpha_1,...,\alpha_n)$ contains a non-empty $S$-adically open, and therefore in particular Zariski-dense, set. We need to ensure that some choice $(\alpha_1,...,\alpha_n)$ in this set leads to a polynomial $f(X) - u  = \prod_{i=1}^n(X-\alpha_i) - u$ as in the proof of Lemma \ref{lem:sn}. However, we there found a Zariski-dense open subset of ``good" values $(\alpha_1,...,\alpha_n)$, and the intersection with our new dense set is then still dense. 

It remains to justify assumption (**). However, (**) is a Zariski-dense open condition on $(\alpha_1,...,\alpha_n)$, and therefore its intersection with the dense set of values $(\alpha_1,...,\alpha_n)$ obtained so far remains dense.
\end{proof}

Note in particular that the special case $G=S_k$ in the above lemma provides many number fields with discriminant of exact maximal exponent $k$, for all $k\in \mathbb{N}$.

The last remaining infinite family is given by the groups $(C_2\wr S_n)\cap A_{2n}$. 
%
%
\begin{lemma}
\label{a2n_intersect}
Theorem \ref{thm:cubefree} holds for the groups $G=(C_2\wr S_n)\cap A_{2n}$.
\end{lemma}
\begin{proof}
We use the construction for $\Gamma=C_2\wr S_n$ in Lemma \ref{lem:wreath}, with the $C_2$-extension $E/\qq(t)$ given by $E=\qq(\sqrt{t})$. The construction in the proof of Lemma \ref{lem:wreath} then yields a $\Gamma$-extension
$\widehat{E}/\qq(u)$ fulfilling the following:
\begin{itemize}
\item[i)] The inertia group at $u\mapsto \infty$ is generated by a $2n$-cycle.
\item[ii)] One more (rational) branch point has a transposition group as inertia group generator, whereas all other inertia group generators are double transpositions.
\item[iii)] $u\mapsto 0$ is unramified in $\widehat{E}$, and its residue extension is unramified at all primes dividing $2n$.
\end{itemize}
In fact, since $E/\qq(t)$ possesses completely split residue extensions, we may (through good choice of the values $\alpha_1,...,\alpha_n$ as in the previous proof) even assume the stronger
\begin{itemize}
\item[iii')] $u\mapsto 0$ is unramified in $\widehat{E}$, and all primes dividing $2n$ split completely in $\widehat{E}_{u\mapsto 0}/\qq$.
\end{itemize}
Now let $L\subset \widehat{E}$ be the fixed field of $G=(C_2\wr S_n)\cap A_{2n}$.
Since only two of the inertia group generators are not contained in $A_{2n}$, the quadratic extension $L/\qq(t)$ is ramified exactly at two rational points (one of them $u\mapsto \infty$). Therefore, $L$  is a rational function field, say $E=\qq(s)$, and without loss $s\mapsto \infty$ is the unique place extending $u\mapsto \infty$. Consequently, the non-trivial inertia groups in the $G$-extension $\widehat{E}/\qq(s)$ are generated by a double-$n$-cycle (at infinity) and double transpositions otherwise. Furthermore, the residue extension at $s\mapsto \infty$ is the same extension $\qq(\zeta_{2n})/\qq$ as the one at $u\mapsto \infty$ (due to total ramification in $\qq(s)/\qq(u)$), and the residue extension at any point $s\mapsto s_0$ extending $u\mapsto 0$ is completely split at every prime dividing $2n$. Now these points $s_0$ are not necessarily rational; they are however $\qq_p$-rational by assumption iii'), and once again using Krasner's lemma, there are then $\qq$-rational values $\widehat{s_0}$ such that $p$ is still completely split in $\widehat{E}_{s\mapsto \widehat{s_0}}/\qq$.
We have therefore verified all assumptions of Corollary \ref{cor:appl} for $\widehat{E}/\qq(s)$, showing the assertion.
\end{proof}

The remaining groups are the four exceptional groups in case (4). These have been dealt with previously (e.g., \cite{Kondo97}). For completeness, we give a proof fitting our criteria.
\begin{lemma}
\label{lem:exceptions}
Theorem \ref{thm:cubefree} holds for the groups $D_5$, $PSL_2(5)$, $PSL_3(2)$ and $AGL_3(2)$.
\end{lemma}
\begin{proof}
All these groups are generated by double transpositions. It therefore suffices to provide for each such group $G$ a $\mathbb{Q}$-regular $G$-extension with all inertia groups generated by double transpositions, and without any prime ramifying in all specializations. Since $PSL_2(5) \cong A_5$, it suffices to treat the same problem for $A_5$. Such extensions are contained, e.g., in \cite{KRS}. For $D_5$-extensions with all inertia groups of order $\le 2$, one may use a $\qq$-regular extension given in \cite{LSWY}.
Degree-$8$ extensions with group $AGL_3(2)$ and all non-trivial inertia groups generated by double transpositions can be extracted from families of polynomials provided in \cite{Malle}, Theorem 6.1. This family is given as
$$f(a,t,X) = X^4(X^2 +aX+2a)(X-2)^2 +t((a-5)(X^2 +X)-2a-2)(X-1)^2,$$
and all non-trivial inertia groups with respect to the parameter $a$ are generated by double transpositions; this still remains true after e.g. specializing $t\mapsto 1$. It then suffices to provide two rational values of $a$ yielding coprime discriminant values, which is easily achieved (e.g. $a=4$ and $a=6$ yield distinct squares of primes as discriminants). Finally, the same source \cite{Malle} also provides suitable extensions for the group $PSL_3(2)$ (see Theorem 4.3), which also has been noted in \cite{KRS}.
\end{proof}

\section{Some related problems}
\label{sec:variants}

\subsection{Almost-squarefree discriminants}
In the list in Proposition \ref{prop:gen_index2}, most groups are generated by the set of all elements of minimal index $\min_{1\ne g\in G} ind(g)$, e.g., double transpositions or $3$-cycles. In such a case, any cubefree discriminant is automatically the square of a squarefree, and therefore ``smallest possible" in the sense of minimizing all the occurring exponents. For the groups $C_2\wr S_n$, the situation is different: they contain transpositions, but in order to generate them, double transpositions are also required. Therefore, a cubefree discriminant for those groups is necessarily of the form  $k^2m$, where $k$ and $m$ are coprime squarefree integers of absolute value $>1$. However, if $k$ is much smaller than $m$, such a number should be considered ``more squarefree" than if $k$ is much larger than $m$. To turn this into a precise notion, consider an integer $N = \prod_{i=1}^r p_i^{e_i}$ and its squarefree part $s(N):=\prod_{p\parallel N}p$. Say that $N$ is {\textit{$\epsilon$-almost-squarefree}}, if $\log(N/s(N)) / \log(N) < \epsilon$.\footnote{It makes sense to consider the expression $\log(N/s(N)) / \log(N)$ for squarefreeness. It is always in $[0,1]$, equal to $0$ exactly for squarefree integers and equal to $1$ exactly for numbers without a simple prime divisor (known as ``squareful numbers").}
An analogous definition of $\epsilon$-almost-$k$-free should be obvious.

To prove the following, we have to slightly improve upon the general construction in Lemma \ref{lem:wreath}.
\begin{theorem}
For any $\epsilon>0$ and $n\ge 2$, there are infinitely many degree-$2n$ extensions $F/\mathbb{Q}$ with group $C_2\wr S_n$ whose discriminant is both cubefree and $\epsilon$-almost-squarefree.
\end{theorem}
\begin{proof}
Take a $C_2\wr S_n$-extension $E/\mathbb{Q}(u)$ as constructed in Lemma \ref{lem:wreath}. This extension has only transpositions and double transpositions as inertia group generators at finite branch points. Furthermore, there are exactly $n-1$ branch points with double transpositions (namely, the  finite branch points of some polynomial map $f(X)-u$). By choosing the $C_2$-extension $E/\mathbb{Q}(t)$ in Lemma \ref{lem:wreath} appropriately, we can ensure that the number of branch points with transposition inertia is larger than any prescribed bound. Indeed, any hyperelliptic cover given by $Y^2 = P(X)$, with a separable polynomial $P$ of degree $D\gg n$ and with a $\mathbb{Q}$-rational point will suffice. Assume additionally that $P$ splits completely over $\qq$. This produces $D$ rational branch points with transposition inertia for $E/\mathbb{Q}(u)$.

Let $F_1(X)\in \zz[X]$ be the product of irreducible polynomials of the transposition branch points. This is of degree $D$ and splits into linear factors $a_iX+b_i$ with $a_i,b_i$ coprime (namely, up to multiplicative constants, exactly the $X-f(\beta)$, where $\beta$ runs through the roots of $P$). Let $F_2(X)$ be the product of irreducible polynomials of the double transposition branch points (of degree $n-1$). 
Applying Theorem \ref{thm:chevalley_weil} and a suitable linear transformation in $u$, we may assume that for all $u_0\in \zz$ the specialization $E_{u\mapsto u_0}/\qq$ is unramified at bad primes and has only transposition and double transposition inertia.
By Beckmann's theorem (Proposition \ref{prop:beckmann}), transposition (resp., double transposition) inertia occurs exactly at the good primes $p$ dividing $F_1(u_0)$ (resp., $F_2(u_0)$) to an odd power. We are therefore done as soon as we can provide a reasonable lower bound for the squarefree part of $F_1(u_0)$. Due to our concrete assumptions on $F_1$, this is quite easy. 

Let $N$ be the largest integer dividing all values $F_1(u_0)$. Then there exist $m_1,m_2\in \zz$ such that
\begin{itemize}
\item[i)] $F_1(m_1X+m_2)$ is divisible in $\zz[X]$ by $N$.
\item[ii)] No prime divides all values of $\tilde{F_1}(X):=F_1(m_1X+m_2)/N$.
\end{itemize}
Indeed, set $m_1$ a power of $N$. For any prime $p$ not dividing $N$, it is obvious that $p$ cannot divide every value of $F_1(m_1+Xm_2)$, by bijectivity of $x\mapsto m_1x+m_2$ modulo $p$. We may therefore restrict to mod-$m_1$ considerations. Multiplying out, it is clear that $F_1(m_1X+m_2)$ is divisible by $N$ for all $m_2$, showing i). Conversely, there must be integers $m_2$ and $x$ such that $F_1(m_1x+m_2)$ is not divisible by any divisor of $m_1$ larger than $N$. Since $F_1(m_1x+m_2) \equiv F_1(m_2)$ mod $m_1$, this then actually holds for all integers $x$ (assuming $m_1$ was chosen large enough), showing ii).

Since $\tilde{F_1}$ splits into linear factors, we may apply Lemma 2.2 in \cite{Ked} to obtain an infinite sequence of $u_0\in A$ such that $\tilde{F_1}(u_0)$ is squarefree; the squarefree part $S(u_0)$ of $F_1(m_1u_0+m_2)$ is then of the same order of growth as $|u_0|^D$. Let $C$ be the product of all bad primes of $E/\mathbb{Q}(u)$. Since every good prime divisor of $S(u_0)$ has transposition inertia in $E_{m_1u_0+m_2}/\qq$, the contribution of all those primes in the discriminant of $E_{m_1u_0+m_2}$ is asymptotically at least proportional to $|u_0|^D/C$ whereas $|F_2(m_1u_0+m_2)| \sim |u_0|^{n-1}$. Choosing first $D$ and then $u_0$ sufficiently large, depending on $\epsilon$, then yields the assertion.
\end{proof}

\subsection{Pairs of fields with the same discriminant}
 A problem closely related to the question which number fields have squarefree discriminant is the question which number fields have the same discriminant as a suitable quadratic number field. This was investigated e.g. in \cite{Kondo}.
 To be tamely ramified with a squarefree discriminant $D$ immediately implies having the same discriminant as some quadratic number field up to sign (namely $\qq(\sqrt{D})$ or $\qq(\sqrt{-D})$).
 In fact, since all non-trivial inertia groups in the corresponding $S_n$-extension are generated by transpositions, this discriminant is up to sign the same as the one of the quadratic subfield fixed by $A_n$. 
 If additionally the $S_n$-extension is totally real (which is easy to ensure using e.g. the construction in Lemma \ref{lem:sn}), then automatically both discriminants are exactly the same (since positive).

 In the same way, being tame with a discriminant $D$ which is a power of an (odd) squarefree integer translates to having the same discriminant as a biquadratic number field $\qq(\sqrt{\pm D_1}, \sqrt{\pm D/D_1})$ as soon as $D_1$ is a non-trivial divisor of $D$. To find (many) such extensions, for all groups which are generated by $3$-cycles and/or double transpositions, we only need to combine Theorem \ref{thm:chevalley_weil} with an argument ensuring that the discriminant of suitable specializations is not a prime power. This is easily achieved, since, if $t_0$ and $t_1$ are two specialization values fulfilling the assertion of Theorem \ref{thm:chevalley_weil} for the given cover $f$, and ramified at primes $p_0$ and at $p_1$ respectively, then some $\{p_0,p_1\}$-adically open set of specializations ramifies at both $p_0$ and $p_1$.

 Other analogs are much more difficult: E.g., a natural question along the same lines is which number fields (of course, with the same Galois groups as before) have the same discriminant as some cyclic cubic field. However, a necessary condition is now that all prime divisors of the discriminant are $\equiv 1$ mod $3$. Via Beckmann's theorem, this leads to deep number-theoretical problems, such as: Given a homogeneous integer polynomial $f$ (to be thought of as the homogenized discriminant of the given $\mathbb{Q}$-regular extension $E/\qq(t)$), does $f$ attain infinitely many values all of whose prime divisors are in some prescribed residue class? 
 

\section{Appendix: The case of $4$-free discriminants}
\label{sec:append}
We conclude by noting that in fact, the methods exhibited in this paper, suffice to provide extensions with $4$-free discriminant for most groups of generator index $3$ as well. Since a fully detailed treatment would be somewhat lengthy without providing many fundamentally new insights, we only sketch the arguments here.

First, one needs to classify groups $G$ of generator index $3$. This still works essentially as in the proof of Proposition \ref{prop:gen_index2}. Indeed, apart from some small permutation degrees (in fact, all exceptions occur in degree $\le 10)$, one may assume that $G$ is imprimitive, $G\le S_k\wr S_n$, and $G$ must project onto $S_n$. Furthermore, $k\le 3$ is the only way to obtain sufficiently small indices for elements projecting to generators of $S_n$. One therefore has (up to finitely many cases) only the following possibilities:
\begin{itemize}
\item[a)] $G=C_3\wr S_n$ or $G=S_3\wr S_n$,
\item[b)] $G$ a low-index subgroup of $C_2\wr S_n$, of $S_3\wr S_n$ or of $C_3\wr S_n$ projecting onto $S_n$. 
\end{itemize}
In fact, due to the irreducibility of the $S_n$-permutation module, the only possibilities in b) are normal subgroups of index $2$ (in the first two cases) and $3$ (in the last case).
While the full wreath products in a) are covered by Lemma \ref{lem:wreath}, the second and third groups in b) can be dealt with by constructions analogous to the one in Lemma \ref{a2n_intersect}. To see this, it is useful to describe these groups somewhat more precisely: the index-$2$ subgroup of $S_3\wr S_n$ is generated by triple transpositions (projecting to transpositions in $S_n$), and double transpositions and $3$-cycles inside $S_3^n$; it does not contain transpositions. In the same way, the index-$3$ subgroup of $C_3\wr S_n$ is generated by triple transpositions. It does not contain $3$-cycles.
Both groups $G$ may then be dealt with by constructing a suitable composition of genus-$0$ extensions with Galois group the full wreath product, and then identifying the fixed field of $G$ as a rational function field, as in the proof of Lemma \ref{a2n_intersect}.

The first group in case b) is generated by $4$-cycles projecting to transpositions in $S_n$. Since it requires some additional ideas, we deal with this class of groups in more detail in the following lemma, thereby reaching altogether the conclusion that all candidate groups with at most finitely many exceptions possess realizations with $4$-free discriminant.
\begin{lemma}
\label{lem:last}
Let $G\le C_2\wr S_n$ be the transitive subgroup generated by $4$-cycles projecting to transpositions in $S_n$.\footnote{Note that this group is indeed strictly smaller than $C_2\wr S_n$, since its intersection with $(C_2)^n$ consists entirely of even permutations!} Then there are infinitely many $G$-extensions with $4$-free discriminant.
\end{lemma}
\begin{proof}
First, let $n$ be odd.
As in Lemma \ref{lem:wreath}, we use a tower $E \supset \qq(t)\supset \qq(u)$, where $\qq(t)/\qq(u)$ is an $S_n$-extension with transposition inertia at all finite primes, and $E/\qq(t)$ is quadratic. This time however, we need the ramified primes of $E/\qq(t)$ to be exactly the primes which are already ramified over $\qq(u)$ (and possibly infinity). To ensure this, let $\qq(t)/\qq(u)$ be a root field of a polynomial $f(X)-u$, where $f\in \zz[X]$ is monic, fulfilling each of the following:
\begin{itemize}
\item[i)] $f$ is separable over $\mathbb{F}_p$ for all prime divisors $p$ of $n$.
\item[ii)] the discriminant of $f(X)-u$ is a separable polynomial in $u$.
\end{itemize}
Of course, i) can be fulfilled by some standard Chinese remainder argument. Furthermore, once again due to the $n$-adically open nature of condition i), one can assume ii) to hold at the same time.
Then, let $E/\qq(t)$ be a splitting field of $X^2 -(-1)^{\frac{n-1}{2}} f'(t)$. Let $\widehat{E}$ be the Galois closure of $E/\qq(u)$. We claim that $E/\qq(u)$ is a $G$-extension fulfilling all assumptions of Corollary \ref{cor:appl}. Firstly, the finite branch points of $E/\qq(t)$ are exactly at the critical points of $f$, i.e., exactly the finite points which are already ramified in $\qq(t)/\qq(u)$. Therefore all non-trivial inertia groups at finite places in $E/\qq(u)$ are generated by $4$-cycles. Furthermore, the residue extension $\widehat{E}_{u\mapsto \infty}/\qq$ equals $\qq(\zeta_n,\sqrt{(-1)^{\frac{n-1}{2}} n}) = \qq(\zeta_n)$, and is in particular ramified only at the prime divisors of $n$.
The residue extension $\widehat{E}_{u\mapsto 0}/\qq$ is generated by all values $\alpha$ and $\sqrt{(-1)^{\frac{n-1}{2}}f'(\alpha)}$, where $\alpha$ runs through the roots of $f$. Due to condition i), the splitting field of $f$ is unramified at all primes dividing $n$.
We claim that the same is even the case after adjoining all roots $\sqrt{(-1)^{\frac{n-1}{2}}f'(\alpha)}$, which then allows application of Corollary \ref{cor:appl}. Indeed, let $p$ be a prime divisor of $n$. Since $\qq(\alpha)/\qq$ is unramified at $p$, $f'(\alpha)$ is of (non-negative) integral $p$-adic valuation, and therefore of valuation $0$, since otherwise $f$ and $f'$ would have a common root modulo $p$, contradicting the fact that $p$ does not divide the discriminant of $f$. Since $p|n$ is odd, this already implies that $p$ remains unramified in the field arising after adjoining all roots $\sqrt{f'(\alpha)}$.
Finally, the Galois group of $E/\qq(u)$ is generated by $4$-cycles projecting to transpositions in $S_n$ and therefore equal to $G$, completing the proof for odd $n$.

For even $n$, we use a similar construction, but with the polynomial $f(X)-u$ replaced by $f(X)-uX$. More precisely, choose $f\in \zz[X]$ monic such that
\begin{itemize}
\item[i)] $Xf(X)$ is separable over $\mathbb{F}_p$ for all prime divisors $p$ of $n-1$.
\item[ii)] $f(0) = (-1)^{n/2}$.\footnote{This is of course compatible with i); just take as the mod-$p$ reduction of $f$ a product of an irreducible with a suitable linear factor.}
\item[iii)] The discriminant of $f(X)-uX$ is a separable polynomial in $u$.
\end{itemize}
Then let $E/\qq(t)$ be a splitting field of $X^2-(-1)^{\frac{n-2}{2}} (tf'(t)-f(t))$.
\footnote{The polynomial $tf'(t)-f(t)$ is the natural replacement of $f'(t)$ in the first case, since its roots are exactly the critical points of the rational function $f(X)/X$.}
The ramification index of $\widehat{E}/\qq(u)$ at $\infty$ is then $n-1$, with exactly the places $t\mapsto 0$ and $t\mapsto \infty$ extending $u\mapsto \infty$ in $\qq(t)$. The residue extension $\widehat{E}_{u\mapsto \infty}/\qq$ then becomes $\qq(\zeta_{n-1}, \sqrt{(-1)^{\frac{n-2}{2}}(n-1)}, \sqrt{(-1)^{n/2}f(0)}) = \qq(\zeta_{n-1})$, by condition ii). Furthermore, since $f(X)$ and $Xf'(X)$ are coprime modulo all $p|n-1$ by i), it follows as in the first case that $\widehat{E}_{u\mapsto 0}/\qq$ is unramified at all $p|n-1$. We may then again apply Corollary \ref{cor:appl}, and since all non-trivial inertia groups at finite places of $\widehat{E}/\qq(u)$ are once again generated by $4$-cycles, the assertion follows also for even $n$.
%
\end{proof}

{\textbf{Acknowledgement}}: I thank Peter M\"uller for some helpful comments on Proposition \ref{prop:gen_index2} and beyond. I thank the Department of Mathematics at KIAS (Seoul) for their hospitality during the time this article was written.

\end{document}